\documentclass[10pt, a4paper]{article}
\usepackage{amsmath, amsfonts}
\usepackage[usenames]{color}
\usepackage[latin1]{inputenc}
\usepackage[table, dvipsnames]{xcolor}
\usepackage{array}
\newcolumntype{P}[1]{>{\centering\arraybackslash}p{#1}}
\newcolumntype{M}[1]{>{\centering\arraybackslash}m{#1}}
\usepackage{multirow}
\usepackage{mathrsfs}
\usepackage{hyperref} 
\hypersetup{bookmarks = true, 
                  colorlinks = true, 
                  linkcolor = blue, 
                  citecolor = blue, 
                  }
\newcommand{\xdownarrow}[1]{%
  {\left\downarrow\vbox to #1{}\right.\kern-\nulldelimiterspace}
}

\newcommand{\xuparrow}[1]{%
  {\left\uparrow\vbox to #1{}\right.\kern-\nulldelimiterspace}
}

\usepackage{graphicx}
\usepackage{amssymb}
\usepackage{amsmath}

\newtheorem{theorem}{Theorem}[section]

\newtheorem{corollary}[theorem]{Corollary}

\newtheorem{definition}[theorem]{Definition}
\newtheorem{example}[theorem]{Example}

\newtheorem{lemma}[theorem]{Lemma}

\newtheorem{proposition}[theorem]{Proposition}
\newtheorem{remark}[theorem]{Remark}
\def\<#1>{\mathinner{\langle#1\rangle}}

\newenvironment{proof}[1][Proof]{\noindent\textbf{#1.} }{\ \rule{0.5em}{0.5em}}

\title{\textbf{On the topology of the transversal slice of a quasi-homogeneous map germ}}
\author{ \ \ \ \\{O.N. Silva\footnote{O.N. Silva: Departamento de Matemática, Universidade Federal de São Carlos, Caixa Postal 676, 13560-905 São Carlos, SP, Brazil, $ \ $ e-mail: otoniel@dm.ufscar.br}}}
\date{}

\begin{document}

\maketitle

\begin{abstract}
We consider a corank $1$, finitely determined, quasi-homogeneous map germ $f$ from $(\mathbb{C}^2,0)$ to $(\mathbb{C}^3,0)$. We describe the embedded topological type of a generic hyperplane section of $f(\mathbb{C}^2)$, denoted by $\gamma_f$, in terms of the weights and degrees of $f$. As a consequence, a necessary condition for a corank $1$ finitely determined map germ $g:(\mathbb{C}^2,0)\rightarrow (\mathbb{C}^3,0)$ to be quasi-homogeneous is that the plane curve $\gamma_g$ has either two or three characteristic exponents. As an application of our main result, we also show that any one-parameter unfolding $F=(f_t,t)$ of $f$ which adds only terms of the same degrees as the degrees of $f$ is Whitney equisingular.
\end{abstract}

\section{Introduction}

$ \ \ \ \ $ Throughout this paper, we assume that $f:(\mathbb{C}^2,0)\rightarrow(\mathbb{C}^3,0)$ is a finite, generically $1-1$, holomorphic map germ, unless otherwise stated. If $f$ has corank $1$, local coordinates can be chosen so that these map germs can be written in the form 

\begin{equation}\label{eq15}
f(x,y)=(x, p(x,y),q(x,y))
\end{equation}

\noindent for some function germs $p,q \in m_2^2$, where $m_2$ is the maximal ideal of the local ring of holomorphic function germs in two variables $\mathcal{O}_2$.

In \cite{slice}, Nuño-Ballesteros and Marar studied the generic hyperplane sections of $f(\mathbb{C}^2)$, usually called, the \textit{transversal slice of $f$} (see Definition \ref{defslice}). They showed that if a certain genericity condition is satisfied, then the transverse slice curve $\gamma_f$ contains some information on the geometry of $f$.

In this work, we consider a corank $1$, finitely determined, quasi-homogeneous map germ $f(x,y)=(x,p(x,y),q(x,y))$ from $(\mathbb{C}^2,0)$ to $(\mathbb{C}^3,0)$. To illustrate the problem that we will present, consider the map germs:

\begin{center}
 $f_1(x,y)=(x,y^3+xy,y^4+3xy)$ $ \ \ $ and $ \ \ $ $f_2(x,y)=(x,y^3,xy+y^5)$,
 \end{center} 

\noindent which are the singularities $P_3$ and $H_2$ of Mond's list \cite[p. 378]{mond7}, respectively. Note that both $f_1$ and $f_2$ are corank $1$, finitely determined and quasi-homogeneous map germs.

Denote by $(X,Y,Z)$ the coordinates of $\mathbb{C}^3$ and consider the hyperplane $H$ given by the equation $X=0$. Note that the hyperplanes sections $H\cap f_1(\mathbb{C}^2)$ and $H\cap f_2(\mathbb{C}^2)$ of $f_1$ and $f_2$ are reduced plane curves parametrized by $u \mapsto (u^3,u^4)$ and $u\mapsto (u^3,u^5)$, respectively. One can check that $H\cap f_1(\mathbb{C}^2)$ is in fact the transversal slice of $f$, in the sense that it satisfies the transversality conditions of \cite[Sec. 3]{slice}. However, in \cite[Ex. 5.2]{slice}, Marar and Nuño-Ballesteros showed that $H\cap g(\mathbb{C}^2)$ is not the transversal slice of $f_2$. This means that the hyperplane $H$ defined by $X=0$ is not a generic plane for $f_2$ in the sense of \cite[Sec. 3]{slice}. On the other hand, there are germs of diffeomorphism $\Phi:(\mathbb{C}^2,0)\rightarrow (\mathbb{C}^2,0)$ and $\Psi: (\mathbb{C}^3,0)\rightarrow (\mathbb{C}^3,0)$ such that $g= \Psi \circ f_2 \circ \Phi$, where 

\begin{center}
$g(x,y)=(x,y^3,xy+xy^2+y^4)$. 
\end{center}
 
In other words, we have that $f_2$ is  $\mathcal{A}$-equivalent to $g$. Furthermore, now the hyperplane $H=V(X)$ is generic for $g$. So, the transversal slice of $g$ is the reduced plane curve parametrized by $u\mapsto (u^3,u^4)$. We conclude that the embedded topological type of the transversal slice of $f_2$ is the same as the plane curve parametrized by $u\mapsto (u^3,u^4)$.

This example shows that for a corank $1$, finitely determined, quasi-homogeneous map germ in the normal form (\ref{eq15}) the plane $H$ defined by $X=0$ may not be generic for $f$. Since the embedded topological type of the transversal slice of $f$ does not depend on the choice of the coordinates (see \cite{pellikaan1}), a solution to fix this inconvenience is to work with the $\mathcal{A}$-equivalence. However, as illustrated in the above example, when we compose a quasi-homogeneous map germ with germs of diffeormorphisms in the source and target we may lose the property of the resulting composite map being quasi-homogeneous (in relation to the new coordinate system). That seems to mean that the embedded topological type of the transversal slice curve of $f$ is not related to the quasi-homogeneous type of $f$. So, the following question is natural:

\begin{flushleft}
\textbf{Question 1:} Let $f:(\mathbb{C}^2,0)\rightarrow (\mathbb{C}^3,0)$ be a corank $1$, finitely determined, quasi-homogeneous map germ on the form (\ref{eq15}), i.e., $f(x,y)=(x,p(x,y),q(x,y))$. Is the embedded topological type of the transversal slice $\gamma_f$ of $f$ determined by the weights of $x$ and $y$ and the weighted degrees of $p$ and $q$?
\end{flushleft}

We know that if $f$ has corank $1$, then $\gamma_f$ is an irreducible plane curve. It is well known that for an irreducible plane curve the characteristic exponents determine and are determined by the embedded topological type of the curve. Thus Question 1 can be reformulated in terms of the characteristic exponents of $\gamma_f$. That is, we can ask if the characteristic exponents of $\gamma_f$ are determined by the weights and degrees of $f$. In the first part of this paper, we present a positive answer to Question $1$, (see Propositions \ref{mult 2} and \ref{mult greater 2}). We show that the number of characteristic exponents of $\gamma_f$ can only be two or three, depending on some relations between the weights and degrees of $f$. More precisely, we show the following result:

\begin{theorem}\label{mainresult1} Let $f:(\mathbb{C}^2,0)\rightarrow (\mathbb{C}^3,0)$ be a corank $1$, finitely determined, quasi-homogeneous map germ. Write $f$ in the normal form

\begin{center}
 $f(x,y)=(x,p(x,y),q(x,y))$,
 \end{center} 

\noindent and let $a,b$ be the weights of the variables $x,y$, respectively. Let $ d_2$ and $d_3$ be the weighted degrees of $p$ and $q$, respectively, with $2 \leq d_2\leq d_3$. Set $c=min\lbrace a,d_2 \rbrace$.  Then:

\begin{flushleft}
 If $a\leq d_2$, $4\leq \dfrac{d_2}{b}$ and $gcd(d_2,d_3)=2$, then $\gamma_f$ has three characteristic exponents given by
\end{flushleft}

 \begin{center}
$d_2$, $ \  $ $d_3$ $ \  $ and $ \  $ $d_2+ d_3 -a$.
 \end{center}

\noindent Otherwise, $\gamma_f$ has only two characteristic exponents given by

 \begin{center}
$\dfrac{d_2}{b}$ $ \ \ $ and $ \ \ $ $ \left( \dfrac{(d_2-c)(d_3-b)\cdot c + (d_2-c)\cdot sab}{ab(d_2-b)} \right) +1$,
 \end{center}

\noindent where

\[ s =   \left\{
\begin{array}{ll}
      0 & if $ the restriction of $f$ to the line $x=0$ is generically $1-$to$-1$ $.\\
      1 & otherwise.   
\end{array} 
\right. \]

\end{theorem}

Clearly when $\gamma_f$ has three characteristic exponents it is not a quasi-homogeneous curve. However, we note that the number $s$ in Theorem \ref{mainresult1} is determined by the weights and degrees of $f$ (see Remark \ref{remarkons}). In this way, Theorem \ref{mainresult1} shows that in fact the embedded topological type of the transversal slice of $f$ is determined by the weights and degrees of $f$, even in the case where the curve $\gamma_f$ is not quasi-homogeneous.

In the second part of this work, we present two natural consequences of Theorem \ref{mainresult1}. More precisely, in Section \ref{sec4} we show that a necessary condition for a corank $1$ finitely determined map germ $g:(\mathbb{C}^2,0)\rightarrow (\mathbb{C}^3,0)$ to be quasi-homogeneous (in a suitable system of coordinates) is that the transversal slice of $g$ must be have either two or three characteristic exponents (Corollary \ref{cor1}). 

Also in Section \ref{sec4}, we show that any one-parameter unfolding $F=(f_t,t)$ of $f$ which adds only terms of the same degrees as the degrees of $f$ is Whitney equisingular (Corollary \ref{cor2}). We also consider some natural questions and provide counterexamples for them. For instance, we show that Question $1$ has a negative answer in corank $2$ case (see Example \ref{excorank2}). We also show that Question $1$ has a negative answer for corank $1$ map germs from $(\mathbb{C}^n,0) \rightarrow (\mathbb{C}^{n+1},0)$ with $n\geq 3$ (see Example \ref{exhigherdim}). We finish presenting examples to illustrates our results, more precisely, we describe the embedded topological type of any quasi-homogeneous map germ of Mond's list (see Section \ref{sec4}).

\section{Preliminaries}

$ \ \ \ \ $ Throughout this paper, given a finite map $f:\mathbb{C}^2\rightarrow \mathbb{C}^3$, $(x,y)$ and $(X,Y,Z)$ are used to denote systems of coordinates in $\mathbb{C}^2$ (source) and $\mathbb{C}^3$ (target), respectively. Also, $\mathbb{C} \lbrace x_1,\cdots,x_n \rbrace \simeq \mathcal{O}_n$ denotes the local ring of convergent power series in $n$ variables. The letters $U,V$ and $W$ are used to denote open neighborhoods of $0$ in $\mathbb{C}^2$, $\mathbb{C}^3$ and $\mathbb{C}$, respectively. For unfoldings, we will use $T$ to denote the parameter space, which is also an open neighborhood of $0$ in $\mathbb{C}$. Throughout, we use the standard notation of singularity theory as the reader can find in Wall's survey paper \cite{wall}.

\subsection{Double point curves for corank 1 map germs}\label{sec2.1}

$ \ \ \ \ $ In this section, we deal only with of corank $1$ map germs. We follow \cite[Sec. 1]{ref7} and \cite[Sec. 3]{mond6}. 

Let $f:(\mathbb{C}^2,0)\rightarrow (\mathbb{C}^3,0)$ be a finite corank $1$ map germ. As we said in Introduction, up to $\mathcal{A}-$equivalence, $f$ can be written in the form $f(x,y)=(x,p(x,y),q(x,y))$. In this case, the lifting of the double point space is defined as:

\begin{center}
$D^2(f)=V \displaystyle \left( x-x',\dfrac{p(x,y)-p(x,y')}{y-y'}, \dfrac{q(x,y)-q(x,y')}{y-y'} \right)$
\end{center}

\noindent where $(x,y,x',y')$ are coordinates of $\mathbb{C}^2 \times \mathbb{C}^2$ and $V(h_1,\cdots,h_l)$ denotes the set of common zeros of $h_1, \cdots, h_l$.

Once the lifting $D^2(f) \subset \mathbb{C}^2 \times \mathbb{C}^2$ is defined, we now consider its image on $\mathbb{C}^2$ by the projection $\pi:\mathbb{C}^2 \times \mathbb{C}^2 \rightarrow \mathbb{C}^2$ onto the first factor, which will be denoted by $D(f)$. For our purposes, the most appropriate structure for $D(f)$ is the one given by the Fitting ideals, because it relates in a simple way the properties of the spaces $D^2(f)$ and $D(f)$. Also, this structure is well behaved by deformations. 

More precisely, given a finite morphism of complex spaces $g:X\rightarrow Y$ the push-forward $g_{\ast}\mathcal{O}_{X}$ is a coherent sheaf of $\mathcal{O}_{Y}-$modules (see \cite[Ch. 1]{grauert}) and to it we can (as in \rm\cite[Sec. 1]{ref13}) associate the Fitting ideal sheaves $\mathcal{F}_{k}(g_{\ast}\mathcal{O}_{X})$. Notice that the support of $\mathcal{F}_{0}(g_{\ast}\mathcal{O}_{X})$ is just the image $g(X)$. Analogously, if $g:(X,x)\rightarrow(Y,y)$ is a finite map germ then we denote also by $ \mathcal{F}_{k}(g_{\ast}\mathcal{O}_{X})$ the \textit{k}th Fitting ideal of $\mathcal{O}_{X,x}$ as $\mathcal{O}_{Y,y}-$module. 

Another important space to study the topology of $f(\mathbb{C}^{2})$ is the double point curve in the target, that is, the image of $D(f)$ by $f$, denoted by $f(D(f))$, which will also be consider with the structure given by Fitting ideals. In this way, we have the following definition:

\begin{definition} Let $f:U \subset \mathbb{C}^2 \rightarrow V \subset \mathbb{C}^3 $ be a finite mapping.

\begin{flushleft}
 (a) Let ${\pi}|_{D^2(f)}:D^2(f) \subset U \times U \rightarrow U$ be the restriction to $D^2(f)$ of the projection $\pi$. The \textit{double point space} is the complex space
\end{flushleft}

\begin{center}
$D(f)=V(\mathcal{F}_{0}({\pi}_{\ast}\mathcal{O}_{D^2(f)}))$.
\end{center}

\noindent Set theoretically we have the equality $D(f)=\pi(D^{2}(f))$.

\begin{flushleft}
(b) The \textit{double point space in the target} is the complex space $f(D(f))=V(\mathcal{F}_{1}(f_{\ast}\mathcal{O}_2))$. Notice that the underlying set of $f(D(f))$ is the image of $D(f)$ by $f$. 
\end{flushleft}

\noindent (c) Given a finite map germ $f:(\mathbb{C}^{2},0)\rightarrow (\mathbb{C}^3,0)$, \textit{the germ of the double point space} is the germ of complex space $D(f)=V(F_{0}(\pi_{\ast}\mathcal{O}_{D^2(f)}))$. \textit{The germ of the double point space in the target} is the germ of the complex space $f(D(f))=V(F_{1}(f_{\ast}\mathcal{O}_2))$.

\end{definition}

\begin{remark}\label{remarkdimdf} If $f:U \subset \mathbb{C}^2 \rightarrow V \subset \mathbb{C}^3 $ is finite and generically $1$-to-$1$, then $D^2(f)$ is Cohen-Macaulay and has dimension $1$ \rm(\textit{see} \rm\cite[\textit{Prop.} \rm 2.1]{ref9})\textit{. Hence, $D^2(f)$, $D(f)$ and $f(D(f))$ are curves. In this case, without any confusion, we also call these complex spaces by the ``lifting of the double point curve'', the ``double point curve'' and the ``image of the double point curve'', respectively.}
\end{remark}

\subsection{Finite determinacy and the invariant $C(f)$}

\begin{definition}(a) Two map germs $f,g:(\mathbb{C}^2,0)\rightarrow (\mathbb{C}^3,0)$ are $\mathcal{A}$-equivalent, denoted by $g\sim_{\mathcal{A}}f$, if there exist germs of diffeomorphisms $\Phi:(\mathbb{C}^2,0)\rightarrow (\mathbb{C}^2,0)$ and $\Psi:(\mathbb{C}^3,0)\rightarrow (\mathbb{C}^3,0)$, such that $g=\Psi \circ f \circ \Phi$.

\begin{flushleft}
 (b) A map germ $f:(\mathbb{C}^2,0) \rightarrow (\mathbb{C}^3,0)$ is finitely determined ($\mathcal{A}$-finitely determined) if there exists a positive integer $k$ such that for any $g$ with $k$-jets satisfying $j^kg(0)=j^kf(0)$ we have $g \sim_{\mathcal{A}}f$.
 \end{flushleft} 

\end{definition}

\begin{remark}\label{remarktriplepoints} Consider a finite map germ $f:(\mathbb{C}^2,0)\rightarrow (\mathbb{C}^3,0)$. By Mather-Gaffney criterion \rm(\cite[\textit{Th}. \rm 2.1]{wall})\textit{, $f$ is finitely determined if and only if there is a finite representative $f:U \rightarrow V$, where $U\subset \mathbb{C}^2$, $V \subset \mathbb{C}^3$ are open neighbourhoods of the origin, such that $f^{-1}(0)=\lbrace 0 \rbrace$ and the restriction $f:U \setminus \lbrace 0 \rbrace \rightarrow V \setminus \lbrace 0 \rbrace$ is stable.} 

\textit{This means that the only singularities of $f$ on $U \setminus \lbrace 0 \rbrace$ are cross-caps (or Whitney umbrellas), transverse double and triple points. By shrinking $U$ if necessary, we can assume that there are no cross-caps nor triple points in $U$. Then, since we are in the nice dimensions of Mather }\rm(\cite[\textit{p}. \rm 208]{mather})\textit{, we can take a stabilization of $f$}, 

\begin{center}
$F:U \times T \rightarrow \mathbb{C}^4$, $F(x,y,t)=(f_{t}(x,y),t)$, 
\end{center}

\noindent \textit{where $T$ is a neighbourhood of $0$ in $\mathbb{C}$. It is well known that the number $C(f):= \sharp $ of cross-caps of $f_t$ is independent of the particular choice of the stabilization and it is also an analytic invariant of $f$} \rm(\textit{see for instance} \rm\cite{mond7}). \textit{One can calculate $C(f)$ as the codimension of the ramification ideal $J(f)$ in $\mathcal{O}_2$, that is:}

\begin{equation}\label{eq18}
C(f)=dim_{\mathcal{C}}\dfrac{\mathcal{O}_2}{J(f)}.
\end{equation}

\end{remark}

We remark that the space $D(f)$ plays a fundamental role in the study of the finite determinacy. In \cite[Th. 2.14]{ref7}, Marar and Mond presented necessary and sufficient conditions for a map germ $f:(\mathbb{C}^n,0)\rightarrow (\mathbb{C}^p,0)$ with corank $1$ to be finitely determined in terms of the dimensions of $D^2(f)$ and other multiple points spaces. In \cite{ref9}, Marar, Nu\~{n}o-Ballesteros and Pe\~{n}afort-Sanchis extended this criterion of finite determinacy to the corank $2$ case. They proved the following result.

\begin{theorem}\rm(\cite{ref9}, \cite{ref7})\label{criterio} \textit{
Let $f:(\mathbb{C}^2,0)\rightarrow(\mathbb{C}^{3},0)$ be a finite and generically $1$ $-$ $1$ map germ. Then $f$ is finitely determined if and only if the Milnor number of $D(f)$ at $0$ is finite.}
\end{theorem}

\subsection{Identification and Fold components of $D(f)$}

$ \ \ \ \ $ When $f:(\mathbb{C}^2,0)\rightarrow (\mathbb{C}^3,0)$ is finitely determined, the restriction $f_{|D(f)}$ of (a representative) $f$ to $D(f)$ is generically $2$-to-$1$ (i.e; $2$-to-$1$ except at $0$). On the other hand, the restriction of $f$ to an irreducible component $D(f)^i$ of $D(f)$ is either generically $1$-to-$1$ or generically $2$-to-$1$. This motivates us to give the following definition which is from \cite[Def. 4.1]{otoniel3} (see also \cite[Def. 2.4]{otoniel1}).

\begin{definition}\label{typesofcomp} Let $f:(\mathbb{C}^2,0)\rightarrow (\mathbb{C}^3,0)$ be a finitely determined map germ. Let $f:U\rightarrow V$ be a representative of $f$ and consider an irreducible component $D(f)^j$ of $D(f)$.

\begin{flushleft}
(a) If the restriction ${f_|}_{D(f)^j}:D(f)^j\rightarrow V$ is generically $1$ $-$ $1$, we say that $D(f)^j$ is an \textit{identification component of} $D(f)$. In this case, there exists an irreducible component $D(f)^i$ of $D(f)$, with $i \neq j$, such that $f(D(f)^j)=f(D(f)^i)$. We say that $D(f)^i$ is the \textit{associated identification component to} $D(f)^j$ or that the pair $(D(f)^j, D(f)^i)$ is a \textit{pair of identification components of} $D(f)$.
\end{flushleft}
\begin{flushleft}
(b) If the restriction ${f_|}_{D(f)^j}:D(f)^j\rightarrow V$ is generically $2$ $-$ $1$, we say that $D(f)^j$ is a \textit{fold component of} $D(f)$.
\end{flushleft}
 
\end{definition}

We would like to remark that the terminology ``fold component'' in Definition \ref{typesofcomp}(b) was chosen by the author in \cite[Def. 4.1]{otoniel3} in analogy to the restriction of $f$ to $D(f)=V(x)$ when $f$ is the cross-cap $f(x,y)=(x,y^2,xy)$, which is a fold map germ. The following example illustrates the two types of irreducible components of $D(f)$ presented in Definition \ref{typesofcomp}.

\begin{example}\label{example}
\textit{Let $f(x,y)=(x,y^2,xy^3-x^5y)$ be the singularity $C_5$ of Mond's list} \rm\cite{mond6}. \textit{Note that $D(f)=V(xy^2-x^5)$. Thus we have that $D(f)$ has three irreducible components given by}

\begin{center}
$D(f)^1=V(x^2-y), \ \ \ $ $D(f)^2=V(x^2+y) \ \ $ and $ \ \ D(f)^3=V(x)$. 
\end{center}

\textit{Notice that $D(f)^3$ is a fold component and $(D(f)^1$, $D(f)^2)$ is a pair of identification components. Also, we have that $f(D(f)^3)=V(X,Z)$ and $f(D(f)^1)=f(D(f)^2)=V(Y-X^4,Z)$ (see Figure \rm\ref{figura1}\textit{)}.}

\begin{figure}[h]
\centering
\includegraphics[scale=0.21]{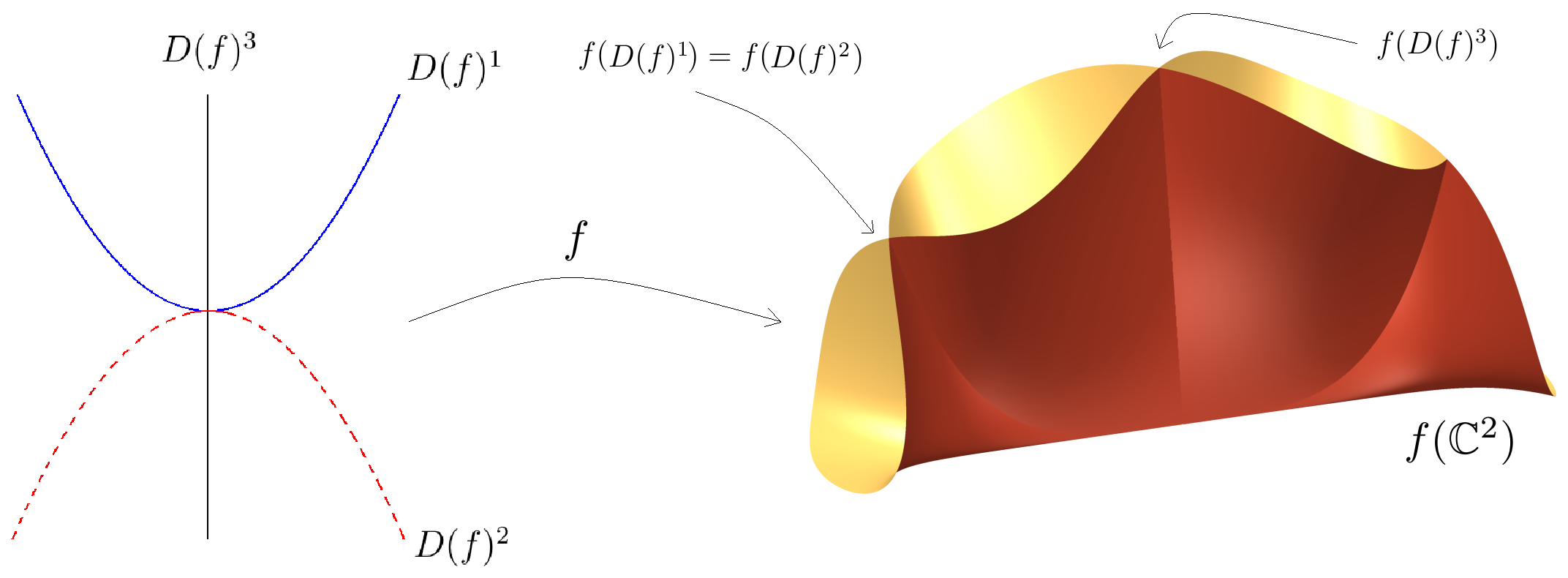}  
\caption{Identification and fold components of $D(f)$ (real points)}\label{figura1}
\end{figure}
\end{example}

\section{The slice of a quasi-homogeneous map germ from $\mathbb{C}^2$ to $\mathbb{C}^3$}\label{sec3}

$ \ \ \ \  $ In \cite{slice}, Marar and Nuño-Ballesteros studied the generic hyperplane sections of $f(\mathbb{C}^2)$ for a map germ map germ $f$ from $(\mathbb{C}^2,0)$ to $(\mathbb{C}^3,0)$ of corank $1$. Following their paper, we present in this section their notion of transversal slice for $f$. 

Let $f:(\mathbb{C}^2,0)\rightarrow (\mathbb{C}^3,0)$ be a corank $1$, finite and generically $1-1$ holomorphic map germ. In this case, the image of the differential $df_0(\mathbb{C}^2)$ is a line in $\mathbb{C}^3$ through the origin. Also, as we said in Remark \ref{remarkdimdf}, the double point space in the target $f(D(f))$ is a curve in $\mathbb{C}^3$.

\begin{definition}\label{defslice} We say that a plane $H_0 \subset \mathbb{C}^3$ through the origin is \textit{generic} for $f$ if the following three conditions hold:

\begin{flushleft}
$ \ \ \ $ (1) $H_0 \cap df_0(\mathbb{C}^2)=\lbrace 0 \rbrace$,\\
$ \ \ \ $ (2) $H_0 \cap f(D(f))=\lbrace 0 \rbrace$, and\\
$ \ \ \ $ (3) $H_0 \cap C_0(f(D(f)))=\lbrace 0 \rbrace$,\\
\end{flushleft}

\noindent where $C_0(f(D(f)))$ denotes the Zariski tangent cone of $f(D(f))$ at $0$. 

$ \ \ $

We remark that the set of generic planes for $f$ is a non-empty Zariski open subset of the Grasmannian of planes of $\mathbb{C}^3$. \textit{In general, the analytic type of the curve $H\cap f(\mathbb{C}^2)$ may depend on the choice of the coordinates and the generic plane $H$, but its embedded topological type does not} \rm(\textit{see} \rm\cite{pellikaan1}). \textit{Hence, for a generic plane $H$ we will denote the plane curve $H\cap f(\mathbb{C}^2)$ by $\gamma_f$ (or $\gamma$, for short, if the context is clear) and it usual to call it the transverse slice of $f$}.
\end{definition}

We would like to study the transversal slice of quasi-homogeneous map germs. Thus, it is convenient to present a  precise definition of this kind of maps.

\begin{definition}\label{defquasihomog} A polynomial $p(x_1,\cdots,x_n)$ is \textit{quasi-homogeneous} if there are positive integers $w_1,\cdots,w_n$, with no common factor and an integer $d$ such that $p(k^{w_1}x_1,\cdots,k^{w_n}x_x)=k^dp(x_1,\cdots,x_n)$. The number $w_i$ is called the weight of the variable $x_i$ and $d$ is called the weighted degree of $p$. In this case, we say $p$ is of type $(d; w_1,\cdots,w_n)$.
\end{definition}

Definition \ref{defquasihomog} extends to polynomial map germs $f:(\mathbb{C}^n,0)\rightarrow (\mathbb{C}^p,0)$ by just requiring each coordinate function $f_i$ to be quasi-homogeneous of type $(d_i; w_1,\cdots,w_n)$, for fixed weights $w_1,\cdots,w_n$. In particular, for a quasi-homogeneous map germ $f:(\mathbb{C}^2,0)\rightarrow (\mathbb{C}^3,0)$ we say that it is quasi-homogeneous of type $(d_1,d_2,d_3; w_1,w_2)$. 

Note that if $f:(\mathbb{C}^2,0)\rightarrow (\mathbb{C}^3,0)$ is a corank $1$ quasi-homogeneous map germ in the normal form (\ref{eq15}), then we can write $p(x,y)=\lambda_1 y^n +x\tilde{p}(x,y)$ and $q(x,y)=\lambda_2y^m+x\tilde{q}(x,y)$, for some $n,m \in \mathbb{N}$, $\lambda_i \in \mathbb{C}$ and $\tilde{p},\tilde{q} \in \mathcal{O}_2$ with $\tilde{p}(x,0)=\tilde{q}(x,0)=0$. This is explained more precisely in the following lemma, where a normal form for $f$ which is more convenient for our purposes is presented. To simplify the notation, in the following we will write $p,q$ in place of $\tilde{p},\tilde{q}$ and set $a:=w_1$ and $b:=w_2$ which will certainly not cause notation confusion.

\begin{lemma}\rm(\cite[\textit{Lemma} \rm 2.11]{otonielformulaJ})\label{lemma corank 1} \textit{Let $g(x,y)=(g_1(x,y),g_2(x,y),g_3(x,y))$ be a corank $1$, finitely determined, quasi-homogeneous map germ of type $(d_1,d_2,d_3; a,b)$. Then $g$ is $\mathcal{A}$-equivalent to a quasi-homogeneous map germ $f$ with type $(d_{i_1}=a,d_{i_2},d_{i_3};a,b)$, which is written in the form}

\begin{equation}\label{eq14} 
f(x,y)=(x, y^n+xp(x,y), \beta y^m+ xq(x,y)),
\end{equation}

\noindent \textit{for some integers $n,m\geq 2$, $\beta \in \mathbb{C}$, $p,q \in \mathcal{O}_2$, $p(x,0)=q(x,0)=0$, where $(d_{i_1},d_{i_2},d_{i_3})$ is a permutation of $(d_1,d_2,d_3)$ such that $d_{i_2}\leq d_{i_3}$.}
\end{lemma}

In the sequel, $\mu(X,0)$ and $m(f(D(f)))$ denotes respectively the Milnor number and the multiplicity of $X$ at $0$. The multiplicity of $f$ is the multiplicity of $f(\mathbb{C}^2)$ at $0$. We note that if $f$ is in the normal form (\ref{eq14}), then $m(f(\mathbb{C}^2))=n$. In this way, we divide the proof of Theorem \ref{mainresult1} into two parts. In the first part we present a proof of the result in the case where the multiplicity of $f$ is $2$ (see Proposition \ref{mult 2}). In the second part we deal with the case of multiplicity greater than $2$ (see Proposition \ref{mult greater 2}). In both cases, we will need the following lemma.

\begin{lemma}\label{milnorgamma} Let $f:(\mathbb{C}^2,0)\rightarrow (\mathbb{C}^3,0)$ be a corank $1$, finitely determined, quasi-homogeneous map germ of type $(d_1=a,d_2,d_3;a,b)$, with $d_2\leq d_3$, and write it as in Lemma \rm\ref{lemma corank 1}, \textit{that is, in the form} 

\begin{center}
$f(x,y)=(x, y^n+xp(x,y), \beta y^m+ xq(x,y))$. 
\end{center}
 
\noindent \textit{Let $\gamma$ be the transversal slice of $f$. Then}

\begin{center}
 $\mu(\gamma,0)=\dfrac{1}{ab^2}\bigg((d_2-b)(d_3-b)c+s a b (d_2-c) \bigg)$
 \end{center} 

\noindent \textit{where $c=min\lbrace a,d_2 \rbrace$, $s=0$ if the restriction of $f$ to the line $x=0$ is generically $1-$to$-1$ or $s=1$, otherwise.}

\end{lemma}

\begin{proof} Since $(f^{-1}(\gamma),0)$ is a germ of smooth curve in $(\mathbb{C}^2,0)$, $\mu(f^{-1}(\gamma),0)=0$. Now the proof follows by \cite[Lemma 5.2]{ref9} and \cite[Th. 3.2]{otonielformulaJ}.\end{proof}

\subsection{The case of multiplicity $2$}

$ \ \ \ \ $ The following result give us a positive answer for Question $1$ in the case where the multiplicity of $f$ is $2$.

\begin{proposition}\label{mult 2} 
Let $f:(\mathbb{C}^2,0)\rightarrow (\mathbb{C}^3,0)$ be a corank $1$, finitely determined, quasi-homogeneous map germ of type $(d_1=a,d_2,d_3;a,b)$, with $d_2\leq d_3$ and suppose that $f$ has multiplicity $2$. Write $f$ as in Lemma \rm\ref{lemma corank 1}, \textit{that is, in the form} 

\begin{center}
$f(x,y)=(x, y^2+xp(x,y), \beta y^m+ xq(x,y))$, 
\end{center}

\noindent \textit{Let $\gamma$ be the transversal slice of $f$. Then $\gamma$ has only two characteristic exponents given by}

\begin{center}
$2$ $ \ \ $ \textit{and} $ \ \ $ $\dfrac{(d_3-b)\cdot c}{ab }+\dfrac{(2b-c)\cdot s}{b}+1$,
 \end{center}

\noindent \textit{where $c=min\lbrace a,d_2 \rbrace$, $s=0$ if the restriction of $f$ to the line $x=0$ is generically $1-$to$-1$ or $s=1$, otherwise.}
\end{proposition}

\begin{proof} Since $f$ has multiplicity $2$, $\gamma$ is a plane curve with multiplicity $2$. Therefore, $\gamma$ has only two characteristic exponents, namely, $2$ and $k$ with $2<k$ odd. By Lemma \ref{milnorgamma} we have that

\begin{equation}\label{eq16}
 k-1=\mu(\gamma,0)= \dfrac{\displaystyle  (d_2-b)(d_3-b)c+sab(d_2-c)}{ab^2} .
\end{equation}

\noindent where $c=min\lbrace a,d_2 \rbrace$ and $s=0$ if $\beta \neq 0$ and $m$ is odd, or $s=1$, otherwise. By expression (\ref{eq16}) we conclude that

\begin{center}
$k=\dfrac{(d_3-b)\cdot c}{ab }+\dfrac{(2b-c)\cdot s}{b}+1$,
\end{center}

\noindent since in this case we have that $d_2=2b$.\end{proof}

$ \ \ $

When we look to the characteristic exponents of $\gamma$ in Theorem \ref{mult 2}, we identify four situations depending on the values that $c$ and $s$ may assume. The following example shows that these four situations can occur.

\begin{example}\label{example1} \noindent \textit{($c=d_2$ and $s=1$) Consider the map germ}

\begin{center}
$f(x,y)=(x,y^2,x^2y-xy^5)$,
\end{center}  

\noindent \textit{which is quasi-homogeneous of type $(4,2,9; 4,1)$. We have that $D(f)=V(x(x-y^4))$ which is reduced. So, by Theorem} \rm\ref{criterio} \textit{we have that $f$ is finitely determined. Using Proposition} \rm\ref{mult 2} \textit{we obtain that the characteristic exponents of $\gamma$ are $2$ and $5$. Note that in this case the plane $H=V(X)$ contains an irreducible component of $f(D(f))$ (see Figure} \rm\ref{figura2}). 

\begin{figure}[h]
\centering
\includegraphics[scale=0.3]{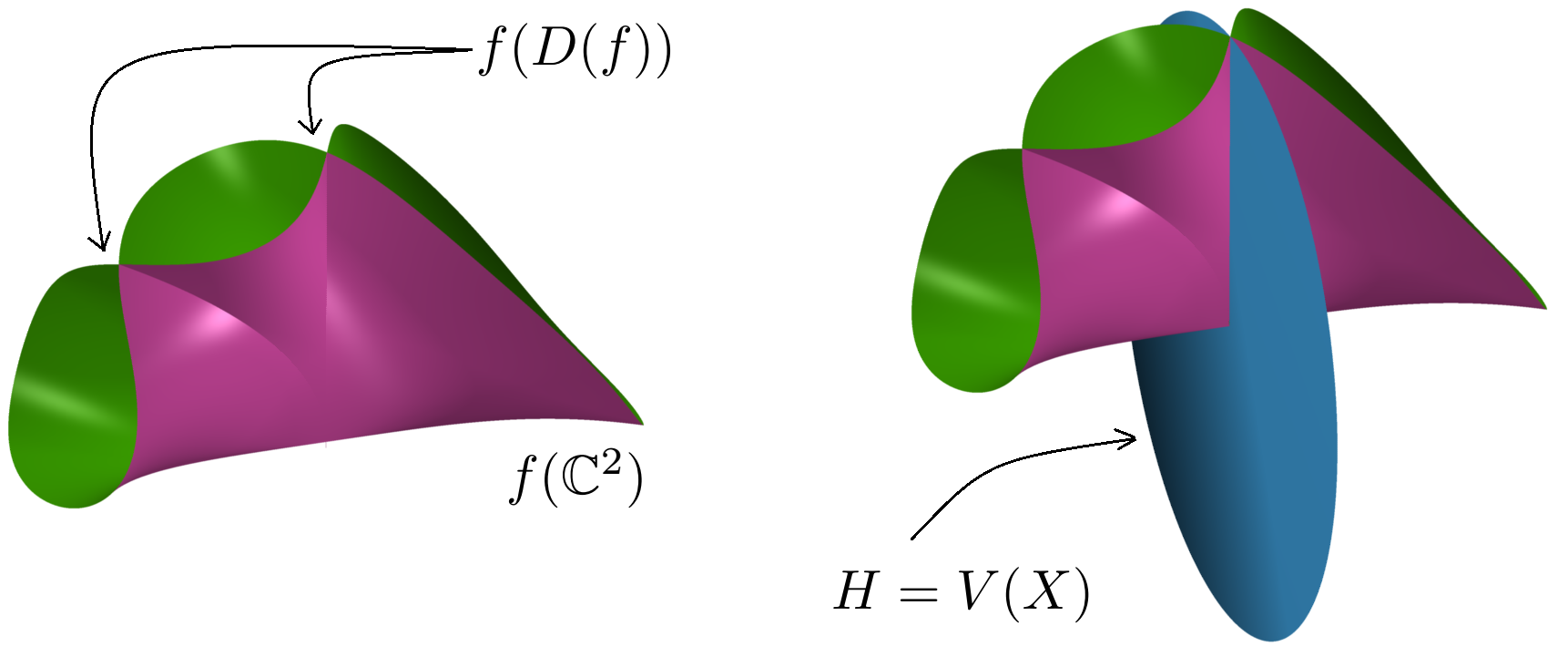} 
\caption{The surface $f(\mathbb{C}^2)$ and the plane $H=V(X)$ (real points)}\label{figura2}
\end{figure}

\textit{We will illustrates the other cases in Table} \rm\ref{tabela1} \textit{(see Section} \rm\ref{sec4}\textit{). For instance, consider the $F_4$-singularity of Mond's list} \rm\cite{mond6}, \textit{which corresponds to the case where $c=a$ and $s=0$. For the case where $c=a$ and $s=1$ consider the cross-cap given by $f(x,y)=(x,y^2,xy)$. Finally, for the case where $c=d_2$ and $s=0$ consider the $B_3$-singularity of of Mond's list.}

\end{example}

\subsection{The case of multiplicity greater than $2$}

$ \ \ \ \ $ In this section we will present a proof of Theorem \ref{mainresult1} in the case where the multiplicity of $f$ is greater than $2$. Note that if $n=m(f(\mathbb{C}^2))\geq 3$ in Lemma \ref{lemma corank 1}, then $\beta\neq 0$. Otherwise, the restriction of $f_{|V(x)}:V(x)\rightarrow \mathbb{C}^3$ of $f$ to the curve $V(x) \subset \mathbb{C}^2$ will be generically $n$-to-$1$ over its image, a contradiction since $f$ is finitely determined. In this case, we will suppose that $\beta=1$. Also, if $f$ is finitely determined, then by \cite[Prop. 1.15]{mondformulas} we have that $D(f)=V(\lambda(x,y))$, where 

\begin{equation}\label{eq2}
\lambda(x,y)=x^{s}y^{v} \cdot  \displaystyle \left( { \prod_{i=1}^{r}}(y^a-\alpha_i x^b) \right)
\end{equation}

\noindent $s,v \in \lbrace 0,1 \rbrace$, $\alpha_i \in \mathbb{C}$, $\alpha_i\neq 0$ for all $i=1,\cdots,r$ and $r=\dfrac{b(n-1)(m-1)-s a - v b}{ab}$. In particular, $\lambda$ is a quasi-homogeneous polynomial of type $(b(n-1)(m-1);a,b)$. To prove the main result of this section we will need the following three lemmas.

\begin{lemma}\label{lemma aux}  Let $f:(\mathbb{C}^2,0)\rightarrow (\mathbb{C}^3,0)$ be a corank $1$, finitely determined, quasi-homogeneous map germ of type $(d_1=a,d_2,d_3;a,b)$, with $d_2\leq d_3$, and multiplicity greater than $2$. Write $f$ as in Lemma \rm\ref{lemma corank 1}, \textit{that is, in the form} 

\begin{center}
$f(x,y)=(x, y^n+xp(x,y), y^m+ xq(x,y))$. 
\end{center}

\noindent \textit{Then:}

\begin{flushleft}
\textit{(a) $V(x)$ is an irreducible component of $D(f)$ if and only if $gcd(n,m)=2$. Furthermore, if $s=1$ in} \rm(\ref{eq2}), \textit{then $V(x)$ is a fold component of $D(f)$.}
\end{flushleft}

\begin{flushleft}
\textit{(b) If $V(x)$ is an irreducible component of $D(f)$, then $b=1$ and $a$ is odd.}
\end{flushleft}

\begin{flushleft}
\textit{(c)  If $v=1$ in} \rm(\ref{eq2}), \textit{then $V(y)$ is an identification component of $D(f)$. Furthermore, if $V(y)$ is an irreducible component of $D(f)$, then $a=1$.}
\end{flushleft}

\noindent \textit{(d) $D(f)=V(\lambda(x,y))$, where}

\begin{equation}\label{eq3}
\lambda(x,y)=x^{s} \cdot  \displaystyle \left( { \prod_{i=1}^{r}}(y^a-\alpha_i x^b) \right)
\end{equation}

\noindent $s \in \lbrace 0,1 \rbrace $, $\alpha_i \in \mathbb{C}$ and $r=\dfrac{b(n-1)(m-1)- sa}{ab}$.

\end{lemma}

\begin{proof} (a) If $gcd(n,m)=2$, then the restriction of $f$ to $V(x)$ is generically $2$-to-$1$ and therefore $V(x)$ is a fold component of $D(f)$. On the other hand, suppose that $V(x) \subset D(f)$ and $gcd(n,m) \neq 2$. Since $f$ is finitely determined, then we should have that $gcd(n,m)=1$. Therefore, $V(x)$ should be an identification component of $D(f)$. Thus, there exist another irreducible component of $D(f)$ that has the same image by $f$ as $V(x)$. By expression (\ref{eq2}), this irreducible component should be either $V(y)$ or $V(y^a-\alpha_i x^b)$ for some $i$. However, note that $f(V(y)) \neq f(V(x))$ and $f(V(y^a-\alpha_i x^b)) \neq f(V(x))$ for all $i$, which is a contradiction.

$ \ \ $

\noindent (b) By expression (\ref{eq2}), we have that $D(f)=V(\lambda(x,y))$, where $\lambda$ is  quasi-homogeneous of type $(b(n-1)(m-1);a,b)$. Thus we can write

\begin{center}
$\lambda(x,y)=a_0y^{(n-1)(m-1)}+a_1x^by^{(n-1)(m-1)-a}+a_2x^{2b}y^{(n-1)(m-1)-2a}+\cdots+a_kx^{\zeta b}y^{(n-1)(m-1)- \zeta a}$,
\end{center}

\noindent where $\zeta$ is the greatest positive integer such that $(n-1)(m-1)- \zeta a\geq 0$.

Since $V(x) \subset D(f)$, then $a_0=0$. Since $f$ is finitely determined, then by Theorem \ref{criterio} we have that $D(f)$ is reduced, and therefore $b=1$ and $a_1 \neq 0$. To see that $a$ should be odd, write $f$ in the form:

\begin{center}


$f(x,y)=(x,y^n+b_1x^by^{n-a}+b_2x^{2b}y^{n-2a}+\cdots+b_{\eta}x^{\eta b}y^{n- \eta a}, \ y^m+c_1x^by^{m-a}+c_2x^{2b}y^{m-2a}+\cdots+c_{\theta}x^{\theta b}y^{m- \theta a})$, 
\end{center}

\noindent where $\eta$ (respectively $\theta$) is the greatest positive integer such that $n-\eta a\geq 1$ (respectively $m- \theta a\geq 1$). By formula (\ref{eq18}) in Remark \ref{remarktriplepoints}, we obtain that $C(f)$ is the codimension of the ideal $ I = \displaystyle \biggl< n \cdot y^{n-1}+ x\cdot\dfrac{\partial p}{\partial y}, \ m \cdot y^{m-1}+ x\cdot\dfrac{\partial q}{\partial y} \biggr> $ in $\mathcal{O}_2$. We conclude that either $b_{\eta} \neq 0$ and $n-\eta a=1$ or $c_{\theta}\neq 0$ and $m-\theta a=1$. Otherwise, the codimension of $I$ in $\mathcal{O}_2$ is not finite and hence $C(f)$ is also not finite, which is a contradiction since $f$ is finitely determined. Now, we have that $n,m$ are both even. Since either $n-\eta a=1$ or $m- \theta a=1$, by parity of $n,m$ and $a$ we conclude that $a$ should be odd.

$ \ \ $

\noindent (c) Note that the restriction of $f$ to $V(y)$ is generically $1$-to-$1$. Thus, if $V(y) \subset D(f)$, then it is an identification component of $D(f)$. In this case, there exist another irreducible component of $D(f)$, which is necessarily on the form $V(y^a-\alpha_i x^b)$, for some $i$ with $\alpha_i \neq 0$, such that $f(V(y))=f(V(y^a-\alpha_ix^b))$. Set $\mathscr{C}_{\alpha_i}:=V(y^a-\alpha_i x^b)$ and consider a parametrization $\varphi_{\alpha_i}: W \rightarrow U $ of  $\mathscr{C}_{\alpha_i}$ defined by $\varphi_{\alpha_i}(u)=(u^a,\rho_i u^b)$, where $W$ is an open neighbourhood of $0$ in $\mathbb{C}$ and $\rho_i \in \mathbb{C}$ is such that $\rho_i^a =\alpha_i$. Since $\mathscr{C}_{\alpha_i}$ is an identification component of $D(f)$, then the mapping $\tilde{\varphi}_{\alpha_i}:= f\circ \varphi_{\alpha_i}: W \rightarrow V$, defined by 

\begin{equation}\label{eq17}
\tilde{\varphi}_{\alpha_i}:=(u^a,\rho_{1,i} u^{d_2}, \rho_{2,i} u^{d_3}),
\end{equation}

\noindent is a parametrization of $f(\mathscr{C}_{\alpha_i})$, for some $\rho_{1,i}, \rho_{2,i} \in \mathbb{C}$. Since $f(V(y))=f(V(y^a-\alpha_ix^b))$, then $\rho_{1,i}=\rho_{2,i}=0$. Since the restriction of $f$ to $\mathscr{C}_{\alpha_i}$ is generically $1$-to-$1$, we conclude that $a=1$.

$ \ \ $

\noindent (d) Suppose that $v=1$ in (\ref{eq2}). By (c), we have that $a=1$. Thus we can rewrite (\ref{eq2}) allowing one of the $\alpha_i's$ to be zero, when $a=1$.\end{proof}

\begin{lemma}\label{lemma aux2} Let $f:(\mathbb{C}^2,0)\rightarrow (\mathbb{C}^3,0)$ be a corank $1$, finitely determined, quasi-homogeneous map germ of type $(d_1=a,d_2,d_3;a,b)$, with $d_2\leq d_3$, and multiplicity greater than $2$. Let $\theta$ be the largest positive integer such that $m-\theta a \geq 1$. If $a>bn$, then $\theta=(m-1)/a$ and $f$ can be written in the following form

\begin{center}
$f(x,y)=(x,y^n,y^m+c_1x^by^{m-a}+c_2x^{2b}y^{m-2a}+\cdots+c_{\theta}x^{\theta b}y)$,
\end{center}

\noindent where $c_1,\cdots,c_{\theta}\in \mathbb{C}$ and $c_{\theta}\neq 0$.

\end{lemma}

\begin{proof} The proof is similar to the proof of Lemma \ref{lemma aux}(b) observing that if $m-\theta a>1$ or $c_{\theta}=0$ then $C(f)$ is not finite, a contradiction since $f$ is finitely determined.\end{proof}

$ \ \ $

The following lemma shows how the Milnor number of an irreducible plane curve with only three characteristic exponents can be calculated.

\begin{lemma} \label{lemma aux3} Let $(X,0)$ be a germ of irreducible reduced plane curve and suppose that it has only three characteristic exponents, denoted by $e_0,e_1$ and $e_2$. Set $e:=gcd(e_0,e_1)$, the greatest common divisor of $e_0$ and $e_1$. Then

\begin{center}
$\mu(X,0)=(e_2-e_1)\cdot e+(e_1-1)\cdot e_0- e_2+1$.
\end{center}  
\end{lemma}

\begin{proof} By \cite[Prop. 4.3.5]{campillo} one can express the generators of the semigroup $\Gamma$ of $(X,0)$ in terms of the characteristic exponents $e_0$, $e_1$ and $e_2$. Since $(X,0)$ is a plane curve, the conductor of the semigroup $\Gamma$ coincides with the Milnor number of $(X,0)$. Now, the proof follows by \cite[Prop. 4.4.5(iii)]{campillo}.\end{proof}

$ \ \ $

Now we are able to give a positive answer for Question $1$ in the case where the multiplicity of $f$ is greater than two. We do this in the following proposition.

\begin{proposition}\label{mult greater 2} Let $f:(\mathbb{C}^2,0)\rightarrow (\mathbb{C}^3,0)$ be a corank $1$, finitely determined, quasi-homogeneous map germ of type $(d_1=a,d_2,d_3;a,b)$, with $d_2\leq d_3$, and multiplicity greater than $2$. Write $f$ as in Lemma \rm\ref{lemma corank 1}, \textit{that is, in the form} 

\begin{center}
$f(x,y)=(x, y^n+xp(x,y), y^m+ xq(x,y))$. 
\end{center}

\noindent \textit{Set $c:=min \lbrace a,d_2 \rbrace$ and let $\gamma$ be the transversal slice of $f$. Then the characteristic exponents of $\gamma$ are given in terms of $a,b, d_2=bn,$ and $d_3=bm$ as follows in Table} \rm\ref{tabela2}. 

\begin{table}[!h]
\caption{Characteristic exponents of the transversal slice of $f$.}\label{tabela2}
\centering
{\def\arraystretch{2.8}\tabcolsep=22pt 

\begin{tabular}{ c || c | c ||  c }

\hline 
\rowcolor{lightgray}
 \textbf{Cases} & \multicolumn{2}{c||}{\textbf{Conditions}} & \textbf{Characteristic Exponents of $\gamma$}  \\
			  
\hline

$Case$ $A$  & \multirow{2}{*}{$a\leq d_2$} & $gcd(n,m)=1$ & $\dfrac{d_2}{b}$, $ \ \dfrac{d_3}{b}$ \\ \cline{1-1} \cline{3-3} \cline{4-4}
     
$Case$ $B$     &  & $gcd(n,m)=2$  & $d_2$, $ \ d_3, \ d_2+d_3-a$ \\
\hline

$Case$ $C$ & \multicolumn{2}{c||}{$a>d_2$} & $\dfrac{d_2}{b}$, $ \ \dfrac{(d_3-b)d_2}{ab}+1$   \\

\hline
\end{tabular}
}
\end{table}

\end{proposition}

\begin{proof} Take a representative $f:U\rightarrow V$ of $f$. By Lemma \ref{lemma aux} (d), we have that $D(f)=V(\lambda(x,y))$, where

\begin{center}
$\lambda(x,y)=\displaystyle { x^{s}\prod_{i=1}^{r}}(y^a-\alpha_i x^b)$,
\end{center}

\noindent $s \in \lbrace 0,1 \rbrace$, $\alpha_i \in \mathbb{C}$ are all distinct and $r=\dfrac{(d_2-b)(d_3-b)-sab}{ab^2}$. 

Set $\mathscr{C}_{\alpha_i}:=V(y^a-\alpha_i x^b)$. As in the proof of Lemma \ref{lemma aux}(c), consider a parametrization $\varphi_{\alpha_i}: W \rightarrow U $ of  $\mathscr{C}_{\alpha_i}$ defined by $\varphi_{\alpha_i}(u)=(u^a,\rho_i u^b)$, where $W$ is an open neighbourhood of $0$ in $\mathbb{C}$ and $\rho_i \in \mathbb{C}$ is such that $\rho_i^a =\alpha_i$. So, if $\mathscr{C}_{\alpha_i}$ is an identification component of $D(f)$, then the mapping $\tilde{\varphi}_{\alpha_i}:= f\circ \varphi_{\alpha_i}: W \rightarrow V$, defined by 

\begin{equation}\label{eq19}
\tilde{\varphi}_{\alpha_i}:=(u^a,\rho_{1,i} u^{d_2}, \rho_{2,i} u^{d_3}),
\end{equation}

\noindent is a parametrization of $f(\mathscr{C}_{\alpha_i})$, for some $\rho_{1,i}, \rho_{2,i} \in \mathbb{C}$. On the other hand, if $\mathscr{C}_{\alpha_i}$ is a fold component of $D(f)$, then the mapping $\varphi_{\alpha_i}': W \rightarrow V$, defined by 

\begin{equation}\label{eq20}
\varphi'_{\alpha_i}(u):=(u^{a/2},\rho_{1,i}' u^{{d_2}/2}, \rho_{2,i}' u^{{d_3}/2}),
\end{equation}

\noindent is a parametrization of $f(\mathscr{C}_{\alpha_i})$, for some $\rho_{1,i}', \rho_{2,i}' \in \mathbb{C}$. Set $c:=min\lbrace a, d_2 \rbrace$. Note that if $a>d_2$, then $\rho_{1,i},\rho_{1,i}'\neq 0$, by Lemma \ref{lemma aux2}. In this way, by expressions (\ref{eq19}) and (\ref{eq20}) we can see that the tangent cone of $f(\mathscr{C}_{\alpha_i})$ is the line with direction given by the vector $(1,0,0)$ if $a<d_2$, $(1,\rho_{1,i},0)$ or $(1,\rho_{1,i}^{'},0)$ if $a=d_2$ and finally $(0,1,0)$ if $a>d_2$.

Set $\mathscr{C}:=V(x)$. If $\mathscr{C} \subset D(f)$, then by Lemma \ref{lemma aux} (a) we have that it is a fold component of $D(f)$. In this case, the map $\varphi :W\rightarrow V$ defined by

\begin{equation}\label{eq4}
\varphi(u)=(0,u^{n/2}, u^{m/2})
\end{equation}

\noindent is a parametrization of $f(\mathscr{C})$, and hence its tangent is the line with direction given by the vector $(0,1,0)$.

$ \ \ $

We are now able to handle all cases. Let's show the simplest cases A and C first.

$ \ \ $

\noindent $\bullet$ \textbf{Case A:} By the considerations above, one can see that in this case the plane $H=V(X)$ satisfies the conditions (1),(2) and (3) of Definition \ref{defslice}. Hence, the restriction of $f$ to the line $x=0$, i.e, the map germ $\varphi(u)=f(0,u)=(0,u^{n},u^{m})$, is a Puiseux parametrization for $\gamma$. Therefore, its characteristic exponents are $n=\dfrac{d_2}{b}$ and $m=\dfrac{d_3}{b}$.

$ \ \ $

\noindent $\bullet$ \textbf{Case C:} By Lemma \ref{lemma aux2} we have that $p(x,y)=0$, and therefore we can write $f$ in the form 

\begin{center}
$f(x,y)=(x,y^n,y^m+c_1x^by^{m-a}+c_2x^{2b}y^{m-2a}+\cdots+c_{\theta}x^{\theta b}y)$,
\end{center}

\noindent where $c_{\theta} \neq 0$ and $\theta=(m-1)/a$. By expressions (\ref{eq19}) and (\ref{eq20}) we can see that the plane $H=V(X)$ is not generic for $f$ since the condition (3) of Definition \ref{defslice} fails, that is, $H$ contains the tangent of $f(\mathscr{C}_{\alpha_i})$ for all $i$. Furthermore, condition (2) also fails if $gcd(n,m)=2$. However, after the linear change of coordinates on the target given by $(X,Y,Z)\mapsto (X-Y,Y,Z)$, we can rewrite $f$ as: 

\begin{center}
$f(x,y)=(x-y^n,y^n,y^m+c_1x^by^{m-a}+c_2x^{2b}y^{m-2a}+\cdots+c_{\theta}x^{\theta b}y)$.
\end{center}

Now, note that the plane $H=V(X)$ (in the new system of coordinates) satisfies the conditions (1),(2) and (3) of Definition \ref{defslice}. Hence, the map germ $\varphi(u)=f(u^n,u)=(0,u^{n},u^{m}+c_1u^{nb+m-a}+c_2u^{2nb+m-2a}+\cdots+c_{\theta}u^{\theta nb+1})$ is a Puiseux parametrization for $\gamma$. Now note that

\begin{center}
$m>nb+m-a>2nb+m-2a> \cdots > \theta nb +1$.
\end{center}

\noindent Note also that $gcd(n,\theta nb+1)=1$, therefore the characteristic exponents of $\gamma$ are $n=\dfrac{d_2}{b}$ and $\theta nb+1= \dfrac{(d_3-b)d_2}{ab}+1$.

$ \ \ $

Finally, let's show the most complicated case.

$ \ \ $

\noindent $\bullet$ \textbf{Case B:} By Lemma \ref{lemma aux}(a) and (b) we obtain that $b=1$, hence we can write $f$ as

\begin{center}
$f(x,y)=(x,y^n+b_1xy^{n-a}+b_2x^{2}y^{n-2a}+\cdots+b_{\eta}x^{\eta}y^{n- \eta a}, \ y^m+c_1xy^{m-a}+c_2x^{2}y^{m-2a}+\cdots+c_{\theta}x^{\theta }y^{m-\theta a})$, 
\end{center}

\noindent where $\eta$ (respectively $\theta$) is the greatest positive integer such that $n- \eta a\geq 1$ (respectively $m- \theta a\geq 1$).

Note that in this case the plane $H=V(X)$ fails to be generic for $f$ only because the condition (2) fails, since $H$ contains the image of $V(x)$, which is an irreducible component of $D(f)$ (see Lemma \ref{lemma aux}(a) and expression (\ref{eq4})).

Consider the one-parameter unfolding $F:(\mathbb{C}^2 \times \mathbb{C},0)\rightarrow (\mathbb{C}^3 \times \mathbb{C},0)$, $F=(f_t(x,y),t)$ where $f_t:(\mathbb{C}^2,0)\rightarrow (\mathbb{C}^3,0)$ is defined as

\begin{center}
$f_t(x,y)=(x-ty^n,y^n+b_1xy^{n-a}+b_2x^{2}y^{n-2a}+\cdots+b_{\eta}x^{\eta}y^{n- \eta a}, \ y^m+c_1xy^{m-a}+c_2x^{2}y^{m-2a}+\cdots+c_{\theta}x^{\theta }y^{m-\theta a})$. 
\end{center}

\textbf{Statement:} We claim that $F$ is Whitney equisingular (see Definition \ref{defWhiteq}).

$ \ \ $

\textit{Proof of the Statement.} We have that $F$ is upper, that is, $F$ adds in the $i$-coordinate function of $f$ only terms of weighted degree greater than or equal to $d_i$. Hence, $F$ is topologically trivial by \cite[Th. 1]{damon92}. By  \cite[Cor. 40]{ref2} (see also \cite[Th. 6.2]{ref1}) we have that $\mu(D(f_t),0)$ is constant.

Let $\mathscr{C}_{\alpha_i,t}$ (respectively $\mathscr{C}_{t}$) be the deformation of $\mathscr{C}_{\alpha_i}$ (respectively, $\mathscr{C}$) induced by $F$. Note that the restriction of $f_t$ to $V(x)$ is generically $2$-to-$1$ for any $t$. This means that $\mathscr{C}_{t}=V(x)$ for any $t$, that is, $F$ induces a trivial deformation of $\mathscr{C}$, hence $m(f_t(\mathscr{C}_t))=n/2$ for any $t$. By expressions (\ref{eq19}) and (\ref{eq20}) we can see that either $m(f(\mathscr{C}_{\alpha_i}))=a$ (if $\mathscr{C}_{\alpha_i}$ is an identification component of $D(f)$) or $m(f(\mathscr{C}_{\alpha_i}))=a/2$ (if $\mathscr{C}_{\alpha_i}$ is a fold component of $D(f)$). Since $F$ is upper we see that the multiplicity of $f_t(\mathscr{C}_{\alpha_i,t})$ should be constant since $F$ adds only terms with degree greater than or equal to the degrees of the coordinate functions of $\tilde{\varphi}_{\alpha_i}$ (and also $\varphi^{'}_{\alpha_i}(u)$). Therefore, the multiplicity of $f_t(D(f_t))$ is constant. By \cite[Prop. 8.6 and Cor. 8.9]{gaffney} we conclude the $F$ is Whitney equisingular, which proves the statement.

$ \ \ $

Note that $H\cap f_t(\mathscr{C})=H \cap C(f_t(\mathscr{C}))=0$ for any $t \neq 0$, where $H=V(X)$ as above. Now, note that  $H\cap f_t(\mathscr{C}_{\alpha_{i,t}})= H \cap C(f_t(\mathscr{C}_{\alpha_{i,t}}))=0$ for any $t$. In fact, since $F$ is Whitney equisingular, $F(D(F))$ is Whitney equisingular. Hence each family of curves $f_t(\mathscr{C}_{\alpha_{i,t}})$ is Whitney equisingular (see \cite[Prop. 4.11(b)]{quartelypaper}). Since $H\cap f(\mathscr{C}_{\alpha_i})= H \cap C(f(\mathscr{C}_{\alpha_i}))=0$ and $f_t(\mathscr{C}_{\alpha_{i,t}})$ is Whitney equisingular, we conclude that for any $t$ sufficiently small we should have that $H\cap f_t(\mathscr{C}_{\alpha_{i,t}})= H \cap C(f_t(\mathscr{C}_{\alpha_{i,t}}))=0$. Therefore, we conclude that the plane $H$ is generic for $f_t$ for any $t\neq 0$ small enough. Now, for $t \neq 0$, a parametrization of the transversal slice $\gamma_{f_t}$ of $f_t$ is given by:

\begin{equation}\label{eq21}
\varphi(u)=(u^n+b_1tu^{n+n-a}+\cdots+b_{\eta}t^{\eta}u^{\eta n+n-\eta a}, \ u^m+c_1tu^{n+m-a}+\cdots+c_{\theta}u^{\theta n + m-\theta a}).
\end{equation}

By Lemma \ref{lemma aux} (b) we obtain that $a$ is odd. Since $n$ is even and $a\leq d_2=n$, we have that $a<n$. Therefore

\begin{center}
$n<m<n+m-a<\cdots$
\end{center}

\noindent and $gcd(n,m,n+m-a)=1$. 

Set $\epsilon:=1+b_1tu^{n-a}+\cdots+b_{\eta}t^{\eta}u^{\eta n- \eta a}=(1+A(u))$ which is an invertible element in $\mathcal{O}_1 \simeq \mathbb{C}\lbrace u \rbrace$. By \cite[Th. 3 and 17]{niven} (see also \cite[Th. 2.2]{cohen}) we know that there exist a unique invertible element $\xi$ in $\mathcal{O}_1$ such that $\xi^{n}=\epsilon$. We denote $\epsilon^{1/n}:=\xi$. More precisely, we have that

\begin{center}
$\epsilon^{1/n}=(1+A(u))^{1/n}= \displaystyle \sum_{j=0}^{\infty}$ $ {1/n}\choose{j}$ $A(x)^j = 1+\dfrac{1}{n}A(x)+\dfrac{\frac{1}{n}(\frac{1}{n}-1)}{2!}A(x)^2+\cdots$.
\end{center}

\noindent where $ {1/n}\choose{j}$ denotes the generalized binomial coefficient. Therefore, 

\begin{equation}\label{eq22}
(\epsilon^{1/n})^{-1}= 1 - \dfrac{1}{n}b_1u^{n-a}+\cdots
\end{equation}

\noindent where $`` \cdots "$ in (\ref{eq22}) denotes the terms of degree strictly greater than $n-a$. 

Consider the isomorphism $\chi:(\mathbb{C},0)\rightarrow (\mathbb{C},0)$ defined by $\chi(u)=u \cdot (\epsilon^{1/n})^{-1}$. Note that 

\begin{equation}\label{eq23}
\varphi \circ \chi(u)= \left( u^n, u^m \left( 1 - \dfrac{1}{n}b_1u^{n-a}+\cdots \right) ^m+c_1tu^{n+m-a} \left( 1 - \dfrac{1}{n}b_1u^{n-a}+\cdots \right)^{n+m-a}+ \cdots \right)
\end{equation}

\noindent is a Puiseux parametrization of $\gamma_t$, for $t\neq 0$. Note also that $m$ is the smallest power (with a non-zero coefficient) that appears in the second coordinate function of $\varphi \circ \chi$ in (\ref{eq23}). Since $gcd(n,m)=2$, we obtain that the characteristic exponents of $\gamma_t$ are $n,m$ and $k$ for some $k>m$ with $gcd(n,m,k)=1$. By Lemmas \ref{milnorgamma} and \ref{lemma aux3} we obtain that

\begin{center}
$\mu(\gamma_t,0)=(n-1)(m-1)+(n-a)=n\cdot m -2m+k-n+1$.
\end{center}

Therefore, $k=n+m-a$. In particular, either $ b_1 \neq 0$ or $ c_1 \neq 0$ in (\ref{eq23}). Now, since $F$ is Whitney equisingular, by \cite[Th. 5.3]{ref9} we have that $ \mu(\gamma_t,0)$ is constant along the parameter space. Hence, the family of plane curves $\gamma_t$ is topologically trivial. Therefore, $\gamma$ and $\gamma_t$, $t \neq 0$, have the same embedded topological type by \cite[Th. 5.2.2]{buch}. Thus we conclude that the characteristic exponents of $\gamma$ are $n=d_2$, $m=d_3$ and $n+m-a=d_2+d_3-a$.\end{proof}

$ \ \ $

Now we are able to present a proof of Theorem \ref{mainresult1}.

$ \ \ $

\textbf{Proof of Theorem \ref{mainresult1}}: Except in the case where $a\leq d_2$, $4\leq \dfrac{d_2}{b}$ and $gcd(d_2,d_3)=2$ (corresponding to the Case C of Proposition \ref{mult greater 2}), we have that $\gamma$ has only two characteristic exponents, namely, $n=d_2/b$ and $k$, where $k$ is described in each case by Propositions \ref{mult 2} and \ref{mult greater 2}. By Lemma \ref{milnorgamma} we have that:

\begin{center}
$\mu(\gamma,0)=(n-1)(k-1)=\dfrac{1}{ab^2}\bigg((d_2-b)(d_3-b)c+s a b (d_2-c) \bigg)$.
\end{center}

This implies that

\begin{center}
$k= \left( \dfrac{(d_2-c)(d_3-b)\cdot c + (d_2-c)\cdot sab}{ab(d_2-b)} \right) +1$
\end{center}

\noindent where $c=min\lbrace a,d_2 \rbrace$, $s=0$ if the restriction of $f$ to the line $x=0$ is generically $1-$to$-1$, or $s=1$, otherwise.

\begin{remark}\label{remarkons} Note that Theorem \rm\ref{mainresult1} \textit{shows that the answer to Question $1$ is in the positive. In fact, we only need to show that the number $s$ is determined by the weights and degrees, at least in the cases where $s$ is fundamental to determine the embedded topological type of the transversal slice of $f$. To see this, one can check that:}

$ \ \ $

\textit{(1) If $m(f(\mathbb{C}^2))\geq 3$, then $s=1$ if $gcd\left( \dfrac{d_2}{b},\dfrac{d_3}{b} \right)=2$ or $s=1$, otherwise} \rm(\textit{see Lemma} \rm\ref{lemma aux}\textit{(a)}\rm).

$ \ \ $

\textit{Now, suppose that $m(f(\mathbb{C}^2))=2$, hence we have that $d_3=as+rab$ and $r \cdot a$ is even, where $r$ is describe in Lemma} \rm\ref{lemma aux} \textit{(d). Then}

$ \ \ $

\textit{(2) If $b\neq 1$, then $s=1$ if $d_3\equiv a (mod \ b)$ or $s=1$, otherwise.}

$ \ \ $

\textit{(3) If $b=1$ and $a$ is odd, then $s=1$ if $d_3\equiv 1 (mod \ 2)$ or $s=1$, otherwise.}

$ \ \ $

\textit{(4) If $b=1$ and $a$ is even, it seems that $s$ may be not determined precisely by a relation between the weight and degrees of $f$. However, in this case Theorem} \rm\ref{mainresult1} \textit{says that the characteristic exponents of $\gamma$ are:}

\begin{center}
$2$ $ \ \ $ \textit{and} $ \ \ $ $\dfrac{d_3-1}{a}+1$.
\end{center}

\textit{Therefore, in this case the embedded topological type of $\gamma$ does not depends on the value of $s$. We conclude that in any situation (except situation (4) which does not depend on $s$) the number $s$ is determined by the weights and degrees of $f$. For instance, the characteristic exponents of $\gamma$ for any corank $1$, finitely determined, quasi-homogeneous map germ of type $(1,4,6; 1,1)$ are $4,6,9$, since in this case $c=s=1$ in Theorem} \rm\ref{mainresult1}.
\end{remark}

We finish this section with an example describing the transversal slice of a corank $1$, finitely determined, quasi-homogeneous map germ of type $(1,4,6; 1,1)$.

\begin{example}\rm(\cite[\textit{Example} \rm 5.5]{otoniel1})  \textit{Consider the map germ $f:(\mathbb{C}^2,0)\rightarrow (\mathbb{C}^3,0)$ defined as} 

\begin{center}
$f(x,y)=(x,y^4,x^5y-5x^3y^3+4xy^5+y^6)$. 
\end{center}

\textit{It is a corank $1$, finitely determined, quasi-homogeneous map germ of type $(1,4,6; 1,1)$. Thus, $c=1$ and $s=1$ and by Theorem} \rm\ref{mainresult1} \textit{we conclude that the characteristic exponents of the transversal slice of $f$ are $4$, $6$ and $9$.}

\end{example}


\section{Some applications, natural questions and examples}\label{sec4}

$ \ \ \ \ $ In this section, we present two natural consequences of Theorem \ref{mainresult1}. We also consider some natural questions and provide counterexamples for them. We finish this section presenting examples to illustrates our results. For the computations we have made use of the software {\sc Singular} \rm\cite{singular} and the implementation of Mond-Pellikaan's algorithm given by Hernandes, Miranda, and Pe\~{n}afort-Sanchis in \rm\cite{ref6}. Alternatively, the reader can consult in \rm\cite[Prop. \rm 4.29]{otoniel3} a description of the presentation matrix of the push-forward module $f_{\ast}\mathcal{O}_{2}$ over $\mathcal{O}_3$ for maps germs in the form $f(x,y)=(x^k,y^n,h(x,y))$.

\subsection{Some applications}

$ \ \ \ \ $ As a direct consequence of Theorem \ref{mainresult1}, we present a necessary condition for a corank $1$ finitely determined map germ to be quasi-homogeneous with respect to some system of coordinates. 

\begin{corollary}\label{cor1} Consider a corank $1$, finitely determined map germ $g:(\mathbb{C}^2,0)\rightarrow (\mathbb{C}^3,0)$ of multiplicity $n$. Then, the following conditions are necessary for the existence of suitable coordinates in which $g$ is quasi-homogeneous.

$ \ \ $

\noindent Condition (a): The transversal slice $\gamma_g$ of $g$ has:

\begin{flushleft}
(1) either two characteristic exponents, namely $n$ and $l$ for some $l>n$, or

$ \ \ $ 

(2) three characteristic exponents, namely, $n<m<k$, with $gcd(n,m)=2$ and $k=n+m-a$ for some $m$ and $a$.
\end{flushleft}

\noindent Condition (b): $\mu(D(g),0)=\tau((D(g),0)$, where $\tau$ denotes the Tjurina number. 
\end{corollary}

\begin{proof} (a) It follows directly by Theorem \ref{mainresult1}.

$ \ \ $  

\noindent (b) It follows by \cite[Prop. 1.15]{mondformulas} (see expression (\ref{eq2}) in Section \ref{sec3}) and Saito's criterion for quasi-homogeneity of isolated hypersurfaces singularities \cite{saito}.
\end{proof}

\begin{example} Consider the map germ $g:(\mathbb{C}^2,0)\rightarrow (\mathbb{C}^3,0)$ defined by

\begin{center}
$g(x,y)=(x,y^8,y^{12}+y^{14}+y^{15}+x^{11}y)$. 
\end{center}

It is a corank $1$ finitely determined map germ. One can check that the transversal slice of $g$ has four characteristic exponents: $8,12,14$ and $15$. Hence, there is no system of coordinates such that $g$ is quasi-homogeneous.

Using {\sc Singular} we found that $\mu(D(g),0)=5978>4575=\tau(D(g),0)$, which is another way to check that $g$ is not quasi-homogeneous.
\end{example}

Another consequence of Theorem \ref{mainresult1}, is about Whitney equisingularity of a one-parameter unfolding of $f$. 

\begin{definition}\label{defWhiteq} Given an unfolding $F(\mathbb{C}^2 \times \mathbb{C},0)\rightarrow (\mathbb{C}^3 \times \mathbb{C},0)$ defined by $F(x,y,t)=(f_t(x,y),t)$, we assume it is origin preserving, that is, $f_t(0,0)=(0,0)$ for any $t$. Hence, we have a $1$-parameter family of map germs $f_t:(\mathbb{C}^2,0)\rightarrow (\mathbb{C}^3,0)$. We say that $F$ is Whitney equisingular if there is a representative of $F$ which admits a regular stratification so that the parameter axes $S=\lbrace (0,0) \times \mathbb{C} \subset \mathbb{C}^2 \times \mathbb{C}$ and $T=\lbrace (0,0,0) \rbrace \times \mathbb{C} \subset \mathbb{C}^3 \times \mathbb{C}$ are strata.

\end{definition}

\begin{corollary}\label{cor2} Let $f:(\mathbb{C}^2,0)\rightarrow (\mathbb{C}^3,0)$ be a corank $1$, finitely determined quasi-homogeneous map germ. Consider an one-parameter unfolding $F:(\mathbb{C}^2 \times \mathbb{C},0)\rightarrow (\mathbb{C}^3 \times \mathbb{C},0)$, $F(x,y,t)=(f_t(x,y),t)$. Write $f_t$ as:

\begin{center}
$f_t(x,y)=(x+g_t(x,y), \tilde{p}(x,y)+h_t(x,y), \tilde{q}(x,y)+l_t(x,y))$.
\end{center}

If $F$ adds only terms of the same degrees as the degrees of $f$, that is, if for a fixed $t_0$ sufficiently small, $w(g_{t_0}(x,y))=w(x)$, $w(h_{t_0}(x,y))=w(\tilde{p}(x,y))$ and $w(l_{t_0}(x,y))=w(\tilde{q}(x,y))$ (considering $t_0$ as a constant in $f_{t_0}$), then $F$ is Whitney equisingular. In particular, $m(f_t(\mathbb{C}^2))$ is constant along the parameter space.
\end{corollary}

\begin{proof} By \cite[Th. 1]{damon92} we have that $F$ is topologically trivial. Hence, by Theorem \ref{criterio} we have that $\mu(D(f_t),0)$ is constant. Note that $f$ and $f_t$ are quasi-homogeneous of the same type. By Theorem \ref{mainresult1} we have that the embedded topological type of the transversal slice $\gamma_t$ of $f_t$ is the same for any $t$. Hence, $\mu(\gamma_t,0)$ is constant and therefore $F$ is Whitney equisingular by \cite[Th. 5.3]{ref9}.\end{proof}

\begin{remark}(a) We note that if $F$ adds some term of degree strictly greater than the degrees of $f$ then it is topologically trivial but it may be not Whitney equisingular (see \rm\cite[\textit{Example} \rm 5.5]{otoniel1}).

$ \ \ $

\noindent \textit{(b) If $g$ has corank $1$ (and it is not necessarily quasi-homogeneous), Marar and Nuño-Ballesteros present in} \rm\cite[\textit{Cor.} \rm 4.7]{slice} \textit{a characterization of Whitney equisingularity of $F$ in terms of the constancy of the invariants $C$, $T$ and $J$ along the parameter space. If $g$ is quasi-homogeneous, Mond shows in} \rm\cite{mondformulas} \textit{(for any corank) that the invariants $C$ and $T$ are determined by the weights and degrees of $g$. The author shows in} \rm\cite[\textit{Th.} \rm 3.2]{otonielformulaJ} \textit{that the invariant $J$ is also determined by the weights and degrees of $f$. Hence, it gives another proof of Corollary} \rm\ref{cor2}.
\end{remark}

\subsection{Natural questions and examples}

$ \ \ \ \ $ We note that Question $1$ also makes sense for quasi-homogeneous finitely determined map germs of corank $2$. More precisely, one can consider the following natural question.

$ \ \ $

\noindent \textbf{Question 2:} Let $f:(\mathbb{C}^2,0)\rightarrow (\mathbb{C}^3,0)$ be a corank $2$ quasi-homogeneous finitely determined map germ. Is the embedded topological type of the transversal slice $\gamma$ of $f$ determined by the weights and degrees of $f$?

$ \ \ $

The following example shows that the answer to Question $2$ is in the negative.

\begin{example}\label{excorank2} Consider the map germs $g_i:(\mathbb{C}^2,0)\rightarrow (\mathbb{C}^3,0)$, defined by

\begin{center}
 $g_1(x,y)=(x^2+xy,y^3,(x+y)^5)$, $ \ \ $ $g_2(x,y)=(x^2-xy+y^2,y^3,(x+y)^5)$, $ \ \ $ and $ \ \ $ $g_3(x,y)=(x^2,y^3,(x+y)^5)$.
\end{center}

Each $g_i$ is a homogeneous finitely determined map germ of corank $2$, of same type, $(2,3,5;1,1)$ (see \rm\cite[\textit{Ex.} \rm 16]{guille} \textit{and} \rm\cite[\textit{Ex.} \rm 5.4]{otoniel1}). \textit{The transversal slice of $g_1$ and $g_2$ has two branches, which we will denoted by $\gamma_{i}^1$ and $\gamma_{i}^2$. On the other hand, the transversal slice of $g_3$ is an irreducible curve, which will denote by $\gamma=\gamma_3^1$.} 

\textit{We recall that topological type of a plane curve determines and is determined by the characteristic exponents of each branch and by the intersection multiplicities of the branches. In this way, Table} \rm\ref{tabela3} \textit{shows that the embedded topological type of these three transversal slices are distinct.}  

\begin{table}[!h]
\caption{Topological invariants for the transversal slice of $g_i$}\label{tabela3}
\centering
{\def\arraystretch{1.7}\tabcolsep=16pt 

\begin{tabular}{ c | c | c |  c }

\hline 

\multirow{2}{*}{\textbf{$g_i(x,y)$}}  & \textbf{Characteristic} & \textbf{Characteristic } & Intersection \\
 
  & \textbf{Exponents of $\gamma_i^1$} & \textbf{Exponents of $\gamma_i^2$} & Multiplicity \\

\hline

$(x^2+xy,y^3,(x+y)^5)$  & $3$, $ \ 5$ & $3$, $ \ 5$ & $15$ \\

\hline

$(x^2-xy+y^2,y^3,(x+y)^5)$  & $3$, $ \ 5$ & $3$, $ \ 5$ & $16$ \\

\hline

$(x^2,y^3,(x+y)^5)$  & $6$, $ \ 10$, $ \ $ $11$ & - & - \\

\hline
\end{tabular}
}
\end{table}

\end{example}  

Fixed the weights and degrees of a corank $2$ finitely determined map germ $f$ from $(\mathbb{C}^2,0)$ to $(\mathbb{C}^3,0)$, it seems that there is a finite number of distinct topological types for $\gamma$. We propose the following problem:

$ \ \ $

\noindent \textbf{Problem:} Fixed the weights and degrees of a corank $2$ finitely determined map germ $f$ from $(\mathbb{C}^2,0)$ to $(\mathbb{C}^3,0)$, determine all possible distinct topological types that the transversal slice of $f$ can have.

$ \ \ $

Since Theorem \ref{mainresult1} does not extend to the corank $2$ case, one might think that the hypothesis of corank $1$ must be a special condition that could be considered in other cases. Thus, another natural question is the following one:

$ \ \ $

\noindent \textbf{Question 3:} Let $f:(\mathbb{C}^n,0)\rightarrow (\mathbb{C}^{n+1},0)$, with $n\geq 3$ be a corank $1$ quasi-homogeneous finitely determined map germ.

$ \ \ $

\noindent(a) Is the embedded topological type of a generic hyperplane section of $f$ determined by the weights and degrees of $f$?

\noindent(b) Is the Corollary \ref{cor1} true in this case? That is, a one-parameter unfolding $F=(f_t,t)$ of $f$ which adds only terms of the same degrees as the degrees of $f$ is Whitney equisingular?

$ \ \ $

The following example shows that the answers to Question $3$ (a) and Question $3$ (b) are in the negative.

\begin{example}\label{exhigherdim} Consider the families of map germs $f_t:(\mathbb{C}^3,0)\rightarrow (\mathbb{C}^4,0)$ defined by

\begin{center}
 $f_t(x,y,z)=(x,y,z^2,z(x^5+yz^{14}+y^{15}+txz^{12}))$.
\end{center}
 
We have that each $f_t$ is a corank $1$ finitely determined map germ of same type, i.e, the deformation of $f_0$ only adds terms of same weighted degrees. However, this family is not Whitney equisingular \rm(\cite[\textit{Ex.} \rm 6.2]{damon92}). \textit{One can check also that the generic hyperplane sections $\gamma_0$ and $\gamma_t$ have distinct embedded topological types.}
\end{example}
  

We finish this work presenting in Table \ref{tabela1} the characteristic exponents for the transversal slice $\gamma$ of each map germ in Mond's list \cite[p.378]{mond6}.

\begin{table}[!h]
\caption{Characteristic exponents for $\gamma$ of quasi-homogeneous map germs in Mond's list.}\label{tabela1}
\centering
{\def\arraystretch{1.8}\tabcolsep=4pt 

\begin{tabular}{ c | c | c | M{1cm}| M{1cm} | M{2.5cm} }

\hline
\rowcolor{lightgray}
Name  &  $f(x,y)$  &   Quasi-Homogeneous type & $c$ & $s$  & Characteristic exponents of $\gamma$  \\

\hline
Cross-Cap     & $(x,y^2,xy)$ & $(1,2,2;1,1)$ & $1$  & $1$  & $2$, $ \ 3$     \\
\hline
$S_k$,  $k\geq 1$ odd   & $(x,y^2,y^3+x^{k+1}y)$  & $(1,k+1,\frac{3(k+1)}{2}; 1,\frac{k+1}{2})$ & $1$ & $0$  &    $2$, $ \ 3$ \\
\hline
$S_k$,  $k\geq 1$ even    & $(x,y^2,y^3+x^{k+1}y)$  & $(2,2k+2,3k+3;2,k+1)$ & $2$  & $0$  &    $2$, $ \ 3$    \\
\hline
$B_k$, $k\geq 3$    & $(x,y^2,y^{2k+1}+x^2y)$  & $(k,2,2k+1;k,1)$ & $2$ & $0$  &    $2$, $ \ 5$      \\
\hline
$C_k$, $k\geq 3$ odd    & $(x,y^2,xy^3+x^ky)$  & $(1,k-1,\frac{3k-1}{2};1,\frac{k-1}{2})$ & $1$ & $1$  &    $2$, $ \ 5$     \\
\hline   
$C_k$, $k\geq 3$ even   & $(x,y^2,xy^3+x^ky)$ & $(2,2k-2,3k-1;2,k-1)$ & $2$ & $1$  &    $2$, $ \ 5$        \\
\hline   
$F_4$     & $(x,y^2,y^5+x^3y)$ & $(4,6,15;4,3)$ & $4$ & $0$  & $2$, $ \ 5$       \\
\hline    
$H_k$     & $(x,y^3,y^{3k-1}+xy)$, $k\geq 2$ & $(3k-2,3,3k-1;3k-2,1)$ & $3$ & $0$  &    $3$, $ \ 4$       \\
\hline
$T_4$     & $(x,y^3+xy,y^4)$ & $(2,3,4;2,1)$ & $2$ & $0$  & $3$, $ \ 4$        \\
\hline  
$P_3$     & $(x,y^3+xy,cy^4+xy^2)^{\ast}$ & $(2,3,4;2,1)$ & $2$ & $0$  &    $3$, $ \ 4$   \\
\hline   

\multicolumn{5}{l}{$\ast \ c\neq 0,1/2,1,3/2$}.

\end{tabular}
}
\end{table}

\begin{remark} We note that all figures used in this work were created by the author using the software Surfer \rm\cite{surfer}.
\end{remark}
 
\begin{flushleft}
\textit{Acknowlegments:} The author would like to thank Juan José Nuño-Ballesteros, João Nivaldo Tomazella and Maria Aparecida Soares Ruas for many helpful conversations. The author acknowledges support by São Paulo Research Foundation (FAPESP), grant 2020/10888-2.
\end{flushleft}

\end{document}